\documentclass[pdftex,letterpaper,12pt]{article}{\twocolumn}
\usepackage[top=1in,bottom=1in,left=1in,right=1in]{geometry}	
\usepackage[dvips]{graphicx}
\usepackage{amsmath, amssymb, amsthm, bbm, mathtools, commath,amsfonts}
\usepackage{color}
\usepackage{datetime}
\usepackage{fancyhdr}
\usepackage{fancyvrb}
\usepackage{hyperref}
\usepackage{pslatex}
\usepackage{verbatim}
\usepackage{lmodern}
\usepackage{tikz-cd}
\usepackage{cleveref}
\usepackage{dsfont}
\usepackage{caption}
\usepackage{enumitem}
\usepackage{helvet}
\usepackage{tipa}

\usepackage[utf8]{inputenc}
\usepackage[english]{babel}

\newtheorem{theorem}{Theorem}

\usepackage[compact]{titlesec}

\title{Metrics and Uniqueness Criteria on the Signatures of Closed Curves}
\date{\today}
\author{Alex Kokot, Ian Klein}

\begin{document}

\maketitle

\section{Introduction}

Determining whether or not two planar curves are congruent under some group action is an
important problem in geometry and has applications to computer vision and image processing.
Calabi, Olver, Shakiban, Tannenbaum, and Haker \cite{Calabi} introduced the paradigm of the differential signature to address this problem. This idea has found applications in various applied problems including medical imaging and automated puzzle assembly [2–4, 8, 12]. 
\\
The origins of the methods go back to Cartan’s solution of the group equivalence problem for submanifolds under Lie group actions [6], however his methodology only provides a local invariant. The differential signature is defined to be the set $S:=\{(\kappa, \kappa_s)\}$, and it is meant to act as a kind of thumbprint, providing an identifier for whether two curves belong to the same equivalence class. The signature is particularly useful when considering closed curves as the set produced is independent of the choice of starting point, and it is often easy to compute $\kappa(t)$ for an arbitrary parameterization, and $\kappa_s(t)$ by the multivariable chain rule. While two congruent curves will always have the same signature, that is, $\Gamma, g\Gamma$ will produce the same signature set, the converse does not always hold. Musso and Nicolodi showed in \cite{Musso} that it is possible to insert sections of constant curvature to an initial curve to produce a 1-parameter family of non-congruent (not equivalent up to group action) curves that have the same signature.
\\
While one approach to limit this failure is to restrict to a particular class of initial curves called non-degenerate curves (see section 3), this paper will attempt to re-imagine the necessary conditions. We will provide a new criteria for uniqueness of the differential signature, paying particular attention to cases such as Euclidean and Affine group actions. We will also work to show when the signature is robust, that is, under what conditions small perturbations of the signature will lead to only small changes in the equivalence class of curves that produce this signature.

\subsection{A Review of Euclidean Curve Reconstruction}
We open by discussing Euclidean curvature, and its basic relationship to the curve it comes from. For this particular differential invariant, it is convenient not only to produce the curvature from a given curve, but also how one can go in the reverse direction; starting with a continuous parameterization of the curvature and finding the curve it corresponds to. We introduce here some basic definitions and formulas that will be assumed in the sequel.\\\\
\textbf{Definition:} The angle of inclination $\theta$ of a curve at a point $p$ (on the curve) is the angle between the $x$-axis and the tangent line at $p$.\\\\
\textbf{Definition:} Let $\gamma(s)$ be the arc length parameterization of a curve. At the point $\gamma(s)$, its euclidean curvature is
$$
\kappa(s) := \theta'(s)
$$
\textbf{Theorem:} Let $h:[a,b] \to \mathbb{R}$ be continuous. There exists a unique curve (up to Special Euclidean congruence) $\gamma$ for which $h(s)$ is the curvature function and $s$ is the arc length parameter.
\begin{proof}
Define $x,y,\theta$ to be functions such that
\begin{align*}
\frac{dx}{ds} &= \cos(\alpha(s)),\\ \frac{dy}{ds} &= \sin(\alpha(s)),\\ \frac{d\theta}{ds} &= h(s).
\end{align*}
We then get the unique solution (given initial conditions $\theta_0, x_0, y_0)$,
\begin{align*}
    \theta(t) &= \theta_0 + \int_0^t h(s)ds,\\
    x(t) &= x_0 + \int_0^t \cos(\alpha(s))ds,\\
    y(t) &= y_0 + \int_0^t \sin(\alpha(s))ds.
\end{align*}
We claim that $\gamma(s):= (x(s), y(s))$ is our desired curve. To do this, we begin by remarking the importance of $\sin$ and $\cos$ in the expression. Recall that $\gamma(s)$ is an arc length parameterization $\Longleftrightarrow$ $|\gamma'| \equiv 1$. Differentiating, we get this immediately as
$$
|\gamma'(s)| = \sqrt{\cos(\theta_s(s))^2 + \sin(\theta_s(s))^2} = 1.
$$
What remains is to show $|\gamma''|\equiv |h|$. Here again, our choice makes this rather easy as
$$
|\gamma''| = \sqrt{x''^2 + y''^2} = \sqrt{|\theta'|^2} = |\frac{d\theta}{ds}| = |h(s)|.
$$
Now, this shows us that $\gamma$ is a curve that works, but is it unique? By our construction, it has initial point $(x_0, y_0)$, and $\theta_0$ is the initial angle of inclination.Thus, if we translate and rotate two curves so that they have the same initial point and angle of inclination, they will have corresponding initial tangent vectors and initial points, so they must be mapped onto each other, thus they are congruent.
\end{proof}

\section{Open Curves}
\subsection{Introduction}

The general flow of our argument will be to show the result on robustness for curves with open (non-periodic) signatures, then use this to show uniqueness for open signatures as a corollary. This argument will not be enough for closed curves as their signatures are necessarily not open as they inherit periodicity from the curve they are derived from. We will extend this result to closed curves in the following section. We will start by considering the Euclidean case in isolation, then we will show how we can apply these results to the general case through methods such as Picard Iterations.

\subsection{Euclidean Signatures}
The Euclidean signature of a curve is the set $\{(\kappa(s),\kappa_s(s))\}$ where $\kappa$ is the Euclidean curvature of the curve and $\kappa_s$ is the derivative with respect to arc length. This set will be invariant to the equivalence class of a curve under the action by the special euclidean group, or $SE(2):= SO(2)\ltimes \mathbb{R}^2$, the rotations and translations of the plane $\mathbb{R}^2.$ \\
We start by introducing some notation that will be convenient to us throughout this paper. We say a \textit{phase portrait} is a set $\{(f(s), f'(s)):\ s\in I\}$ where $I$ is some interval in the domain of $f$. We then say that a parameterization $\sigma(s) = (u(s), v(s))$ is \textit{in phase} if $v(s) = u_s(s)$. Additionally, to distinguish that we are addressing phase-portrait, we will refer to the Cartesian plane as the $u,v$ plane, rather than the traditional $x,y$.\\

\includegraphics[width=0.8\columnwidth]{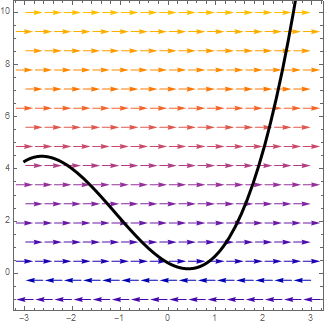}
\captionof{figure}{A curve plotted in the $u$-$v$ plane. We can think of this plane as the vector field $(y,0)$ as depicted above, where the lighter colored arrows correspond to a greater magnitude.}

\medskip

We also introduce specific notation to discuss signatures as phase portraits. A signature is, of course, defined by the curve it is derived from (and often vice versa). Thus we develop notation that relates these two objects. If $S$ is the signature of $\Gamma$, then we say that $\Gamma$ \textit{defines} $S$, or that it is the \textit{defining curve} of $S$. As we will specifically be going from the signature to $\Gamma$, we will sometimes use the notation $S_\Gamma$ in place of $S$ to indicate that $\Gamma$ defines $S_\Gamma$. We will reserve $\gamma(s)$ as a parameterization of $\Gamma$, while $\sigma(s)$ will be used for the in phase parameterization of $S_\Gamma$. \\

Interpreting a signature as a phase portrait gives us some advantages. We first observe that if $\sigma(s)$ is in phase, then $s$ is the arc length parameter of $\Gamma$. If $\sigma$ has domain $[0,L]$, then its defining curves have length $L$, and vice versa (shifting the domain of $\sigma$ appropriately). Also, as $\sigma(s) = (\kappa(s), \kappa_s(s))$, the $v$-coordinate $\kappa_s(s)$ gives us information on how quickly the $u$-coordinate $\kappa(s)$ is changing with respect to $s$.  In this document, we will typically refer to $\kappa_s$ as just simply $\kappa'$ as we will always take our signatures to be in phase parameterizations.\\
We now introduce the following set theoretical notation that will come up in the paper. Note that $| -|$ will always refer to the Euclidean metric in the appropriate dimension. If $S_1,S_2\subset \mathbb{R}^2$, then we define \textit{the distance between} $S_1,S_2$, or their \textit{Hausdorff} distance as
\begin{gather*}
d(p,S_1) = \inf_{q\in S_1} |p-q|\\
d(S_1,S_2) = \max\{\sup_{p\in S_1}d(p, S_2), \sup_{q\in S_2}d(q, S_1)\}.
\end{gather*}

\includegraphics[width=\columnwidth]{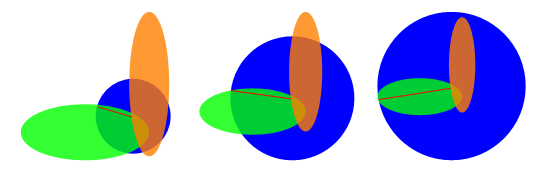}
\captionof{figure}{The Hausdorff distance between two ellipses restricted to the blue circle. As the circle expands, the length of the red line approaches their true Hausdorff distance. }

\medskip

We say that the $\delta$ \textit{neighborhood} of $S\subset \mathbb{R}^n$ is the set
$$
S_\delta := \{p \in \mathbb{R}^n:\ d(p,S)<\delta\}.
$$
Finally, we introduce our own notation for a particular neighborhood about open curves. Let $S$ be an open curve with parameterization $\sigma(s)=(u(s),v(s))$, and end points $s_\ell:=(u_\ell, v_\ell) = \sigma(0)$ and $s_r:=(u_r,v_r) = \sigma(L)$. Then we call the \textit{inner-tube of radius} $\delta$ about $S$ the set
{\footnotesize
$$
IT(S,\delta):=\{(u,v)\subset \mathbb{R}^2:\ \exists s,\ u = u(s),\ |v- v(s)| < \delta\}
$$
}
and the \textit{tube of radius} $\delta$ about $S$ the set

\begin{gather*}
T(S,\delta) := \\
\{p=(u,v)\in \mathbb{R}^2\:\ u<u_\ell,||p-s_\ell||_\infty<\delta\}\\
\cup \{p=(u,v)\in \mathbb{R}^2\:\ u>u_r,||p-s_r||_\infty<\delta\} \\
\cup IT(S,\delta).
\end{gather*}
where here we use the sup-norm $||(u,v)||_\infty = \max\{|u|,|v|\}$. You can think of these as a rectangle about the endpoints of height $2\delta$ and width $\delta$.

\includegraphics[width=\columnwidth]{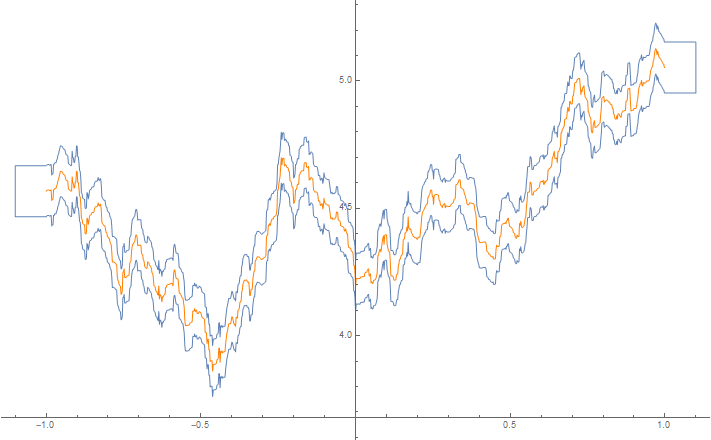}
\captionof{figure}{The tube neighborhood about a curve in $\mathbb{R}^2.$}

\medskip

The above figure illustrates how different the tube neighborhood can be from the $\delta$ neighborhood, as it is much more jagged than what we typically envision. The following lemma details how we can relate these distinct sets.
\newtheorem{lemma}{Lemma}
\begin{lemma}\label{tubetoUni}
If $S$ is the graph of a continuous function over a compact interval, then for all $\delta>0$, there exists $\delta^*>0$ such that $S_{\delta^*} \subset T(S,\delta)$.
\end{lemma}
\begin{proof}
We first discuss why we include the first part of the claim. To say that $\sigma(s):=(u(s), v(s))$ is the graph of a function, it is only necessary that $u(s)$ be an increasing function. This is a sufficient and necessary condition for $T(S,\delta)$ to be a neighborhood of $S$. It is well known that for any closed set $C$, the function $d(x,C)$ as defined above is continuous in $x$. Thus, defining $C:= \mathbb{R}^2\setminus T(S,\delta)$, we see that $m:=\min_{x\in S}d(x,C)$ is achieved at some point $p\in S$. We also have that $d(p,C)>0$ as $T(S,\delta)$ is a neighborhood about $S$. Thus we see that if we take $\delta^*:=m$, then $S_{\delta^*}\subseteq T(S,\delta)$ as desired.
\end{proof}
Because
$$
S_{\delta^*}\subset T(S,\delta).
$$
then if we take $d(S,S')<\delta^*$, we will have
$$
S' \subset S_{\delta^*} \subset T(S,\delta).
$$
Thus if we can prove something about the signatures that are in $T(S_\Gamma,\delta)$, we have shown this property holds for signatures that are within $\delta^*$ of $S_\Gamma$ by the Hausdorff metric. Our argument for the following is much more succinct in the former of the two options, so that is where we keep most of our attention. We aim to prove the following.

\begin{theorem}\label{EuclideanMainTheor}
{
Let $S_\Gamma$ be a signature defined by $\Gamma$, where $\Gamma$ has no vertices and has arc length parameterization $\gamma:[0,L]\to \mathbb{R}^2.$ For all $\varepsilon>0$, there exists $\delta>0$ such that for any $\Gamma^*$ with signature $S_{\Gamma^*}$, if $d(S_{\Gamma^*}, S_\Gamma)<\delta$ then $d(\Gamma,g\Gamma^*)<\varepsilon$ for some $g\in SE(2)$.
}
\end{theorem}
 
 This will be an immediate consequence of the following two propositions, which we devote the next two sections to proving.
 
 \newtheorem{proposition}{Proposition}
 
 \begin{proposition}
Let $\kappa(s):[0,L]\to \mathbb{R},\kappa^*(s):[0,L^*]\to \mathbb{R}$ be curvature functions parameterized with respect to the arc length of their defining curves $\Gamma,\Gamma^*$. Then for all $\varepsilon>0$, there exists $\delta>0$ so that if $|\kappa(s) - \kappa^*(s)| < \delta$ for all $s\in [0,\min\{L,L^*\}]$  and $|L-L^*|<\delta$ then there exists $g\in SE(2)$ such that $d(\Gamma, g\Gamma^*)<\varepsilon$. 
 \end{proposition}
 
 \begin{proposition}
 For any phase portrait $(f(s),f'(s))$, $s\in [0,L]$ and $\varepsilon>0$,
there exists $\delta>0$ such that whenever another phase portrait
 $(g(s),g'(s))$, $s\in[0,L^*]$ has the property $d(S^*,S)<\delta$, we have $|L-L^*|<\varepsilon$ and for $s\in [0,\min\{L,L^*\}]$, $|f(s)-g(s)|<\varepsilon$.
 \end{proposition}

\subsubsection{Criteria for Curvatures}

We first approach this problem by looking at curvatures, and from there we will attempt to reduce the statement on signatures to a statement about their curvatures.

\begin{lemma}\label{closeCurva}
Let $\kappa:[0,L]\to\mathbb{R},\kappa^*:[0,L]\to \mathbb{R}$ be curvature functions parameterized with respect to arc length. For all $\varepsilon>0$, there exists $\delta>0$ where if $|\kappa(s) - \kappa^*(s)| < \delta$ for all $s\in [0,L]$, then $d(\Gamma, g\Gamma^*)<\varepsilon$ for some $g\in SE(2)$. 
\end{lemma}
\begin{proof}
Apply $g$ so that $\Gamma,g\Gamma^*$ have the same initial point and unit tangent vector. If $\gamma(s),\gamma^*(s)$ are the corresponding arc length parameterizations, it is well established that for $\theta_s(s):= \kappa(s)$, we have $\gamma(s) - \gamma(0) = \int_0^s e^{i\theta(s)}dt$, and likewise for its counterpart $\gamma^*$. Thus
\begin{gather*}
    |\gamma(t) - \gamma^*(t)|\\
    =|\int_0^s e^{i\theta(t)}-e^{i\theta^*(t)}dt|\\
    \intertext{As $e^{i\theta}$ is Lipschitz with Lipschitz constant 1,}
    \leq \int_0^s |\theta(t) - \theta^*(t)| dt\\
    \leq \int_0^s |\int_0^t \kappa(h)dh - \int_0^t \kappa^*(h)dh| dt\\
    \leq \int_0^s\int_0^t |\kappa(h) - \kappa^*(h)|dhdt\\
    \leq \int_0^st \delta dt = s^2\delta\leq L^2 \delta/2.
\end{gather*}
Thus if we take $\delta<2\varepsilon/L^2$, the result follows.
\end{proof}
We now want to prove a more general formulation of Proposition 1. The phrasing here is more intimidating, but what it asserts is that it is sufficient that the common domains just take up a sufficiently large portion of the domains of $\kappa,\kappa^*$. The reason that this is sufficient comes from the following lemma.
\begin{lemma}\label{wiggleRoom}
Let $\gamma$ be parameterized with respect to arc length. Then $|\gamma(a) - \gamma(b)|\leq |a-b|.$
\end{lemma}
\begin{proof}
Take $a\leq b$ without loss of generality. Then
\begin{gather*}
    |\gamma(b) - \gamma(a)| = |\int_a^b e^{i\theta(s)}ds|\\
    \leq \int_a^b|e^{i\theta(s)}|ds \leq b-a
\end{gather*}

\end{proof}

\begin{lemma}\label{uncommonDomain}
Let $\kappa(s):[0,L]\to \mathbb{R},\kappa^*(s):[0,L^*]\to \mathbb{R}$ be curvature functions parameterized with respect to arc length. Let $I:=[x_1,y_1]\subset [0,L]$ and $I^*:=[x_2,y_2]\subset [0,L^*]$ be intervals of the same length $y_1-x_1=y_2-x_2 = \ell$. For all $\varepsilon>0$, there exists $\delta>0$ so that if $|\kappa(s+x_1) - \kappa^*(s+x_2)| < \delta$ for all $s\in [0,\ell]$  and $\max\{d([0,L], I), d([0,L^*], I^*)\}<\delta$ then $d(\Gamma, g\Gamma^*)<\varepsilon$ for some $g\in SE(2)$. 
\end{lemma}
\begin{proof}
 By lemma \ref{closeCurva}, we see that the claim is satisfied for $\gamma(I)$ and $\gamma^*(I^*)$. In particular, we choose $\delta$ so that, taking the same $g$ as in lemma \ref{closeCurva}, {\small $$d(\gamma(I), g\gamma^*(I^*))<\varepsilon/2.$$} By the argument presented previously, we particularly take it so that $|\gamma(y_1) - \gamma^*(y_2)|,|\gamma(x_1) - \gamma^*(x_2)|<\varepsilon/2$. We now check the remaining portions of the curves. We first observe that because $d(I,[0,L]),d(I^*,[0,L^*])<\delta$, 
 $$
 |x_1|,|x_2|,|L-y_1|,|L^*-y_2|<\delta.
 $$
 thus by lemma \ref{wiggleRoom}, for all $s\in [0,x_1]$, 
{\small
\begin{gather*}
    |\gamma^*(x_2) - \gamma(s)| \leq |\gamma^*(x_2) - \gamma(x_1)| + |\gamma(x_1) - \gamma(s)|\\
    \leq \varepsilon/2 + |x_1-s| \leq \varepsilon/2 + \delta<\varepsilon
\end{gather*}
}
if we take $\delta < \varepsilon/2.$ The other cases follow similarly. 
\end{proof}

\subsubsection{The Signature as a Phase Portrait}
Our main task now is to show that $\kappa$ and $\kappa^*$ are close given that their signatures are close under the Hausdorff metric. We would expect this to be true as they \textit{start} and \textit{end} close to each other, and go at about the same pace as they have similar $v$-coordinates. To show this formally, we want to construct a ``fastest" and ``slowest" curvature that correspond to signatures in a sufficiently small neighborhood about $S$.\\
Define the \textit{top and bottom curves} of $T(S,\delta)$ to be the functions
\begin{gather*}
\sigma^+(s):=(\kappa(\rho^+(s)), \kappa'(\rho^+(s)) + \delta)\\
\sigma^-(s):=(\kappa(\rho^-(s)), \kappa'(\rho^-(s)) - \delta)
\end{gather*}

\includegraphics[width=\columnwidth]{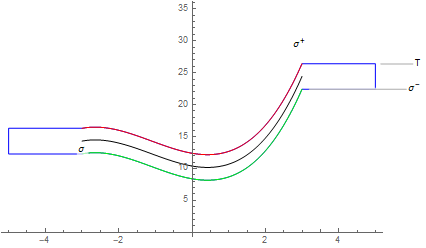}
\captionof{figure}{The tube neighborhood about a signature with $\sigma^+,\sigma^-$ highlighted in red and green respectively.}

\medskip

where $\rho^\pm$ is the parameterization so that they are in phase, and $S=\{\sigma(s)\} = \{\kappa(s),\kappa'(s)\}$ is a phase portrait. Our intuition is that $\rho^+$ will define a slight ``speed-up" in the input as $\sigma^+$ is raised in the $v$-direction, and $\rho^-$ will be a slight ``slow-down" for the opposite reason.\\
A slight disclaimer about this language is that it is only reflective of the case where $T(S,\delta)$ does not intersect the $u$-axis. Otherwise, the curve does not slow-down, but rather changes direction, as the derivative of the curvature changes. For this reason, an important constant will be the distance from the $u$-axis,

$$
m:= \min |\kappa'|.
$$

\begin{lemma}\label{rho}
 The parameterizations $\rho^\pm$ defined above exist so long as $\delta<m$. In this case, we also have that $$s(1-\delta/m) \leq \rho^\pm(s) \leq s(1+\delta/m).$$
\end{lemma}
\begin{proof}
As $\sigma^\pm$ are in phase, we get the defining ODE
{\small
\begin{gather*}
(\kappa(\rho^\pm(s)))' = \kappa'(\rho^\pm(s))(\rho^\pm)'(s) = \kappa'(\rho^\pm(s)) \pm \delta\\
\Longrightarrow (\rho^\pm)'(s)=1\pm \delta /\kappa'(\rho^\pm(s)).
\end{gather*}
}
Because the initial condition is $\rho^\pm(0) = 0$, the second claim immediately follows assuming existence. Define 
$$
H^\pm(s):= 1\pm \delta/\kappa'(s),
$$
so that we now have the separable system of equations
$$
\frac{d\rho^\pm}{ds} = H^\pm(\rho^\pm).
$$
$H^\pm$ is continuous and non-zero by our choice of $\delta$, so we have the unique solution given by
$$
\frac{d\rho^\pm}{H(\rho^\pm)} = ds.
$$
\end{proof}

This argument paves the way for a method we will need to relate a curvature to its signature in a more general setting. We use these ideas to establish a relationship between the signature and its interpretation as the graph of a function.

\begin{lemma}\label{sigDEQ}
Let $F:\mathbb{R}\to \mathbb{R}$ be non-zero and continuous. Then there exists a parameterization $\sigma(s)=(u(s),F(u(s))$ of its graph that is in phase, that is
$$
u'(s) = F(u(s)).
$$
\end{lemma}
\begin{proof}
We start by taking the above differential equation, and solve to get
\begin{gather*}
    \frac{u'(s)}{F(u(s))} = 1 \Longrightarrow \int_0^t \frac{u'(s)dt}{F(u(s))} = t\\
    \Longleftrightarrow \int_{u(0)}^{u(s)}\frac{du}{F(u)} = t.
\end{gather*}
Thus the function
$$
G(s):= \int_{u(0)}^t \frac{du}{F(u)}
$$
has the property $u(G(s)) = G(u(s)) = t$, that is, $G=u^{-1}(s).$ This is exactly as we should expect as, assuming the existence of $u$, we get
{
\begin{gather*}
u'(s) = F(u(s))\\
\Longrightarrow (u^{-1})'(s) = \frac{1}{F(u(u^{-1}(s)))}=\frac{1}{F(s)}.
\end{gather*}
}
As $F$ is non-zero and continuous, $G$ is differentiable and strictly increasing/decreasing. Thus we can invert to get $u$ as desired.
\end{proof}

With this in mind, we can now make precise the idea that the top and bottom curves are the fastest/slowest. That is, what we want to show is that given any curvature function, if we take the curvature that corresponds to the signature at the top of the tube it will have a faster rate of change, while that on the bottom will have a slower rate of change.

\begin{lemma}\label{BottomTop}
Let $S^*\subset IT(S,\delta)$, where $0<\delta<m$. Then, for $t$ such that $\rho^+(s)\leq L$ (in other words, $t\leq (\rho^+)^{-1}(L)$),
$$
\kappa(\rho^-(s+a)) \leq \kappa^*(s)\leq \kappa(\rho^+(s + b))
$$
where $a,b$ are chosen so that $\kappa(\rho^-(a)) = \kappa(\rho^+(b)) = \kappa^*(0).$ In particular, if $\kappa^*(0)=\kappa(0)$ then
$$
\kappa(\rho^-(s)) \leq \kappa^*(s)\leq \kappa(\rho^+(s))
$$
\end{lemma}

\begin{proof}
We first show that such an $a,b$ exist. As has been discussed, $\kappa,\rho^{\pm}$ are monotone and invertible, thus examining the desired values we see that
\begin{gather*}
    a= (\rho^{-})^{-1}(\kappa^{-1}(\kappa^*(0)))\\
     b= (\rho^{+})^{-1}(\kappa^{-1}(\kappa^*(0)))
\end{gather*}
so it remains to show that this is a well defined formula. $S^*\subset IT(S,\delta)$ implies that $\kappa^*(s) \in \kappa([0,L])$ thus the inverse is defined on this set. $\rho^\pm$ have image $[0,L]$ which is the domain of $\kappa$, so the claim follows. Note that these expressions make the second claim immediate as in that case $a=b=0.$
\\
Define
$$
\kappa^-(s):=\kappa(\rho^-(s+a)),\ \kappa^+(s):=\kappa(\rho^+(s+ b)).
$$
As before, we re-interpret these signatures as the graphs of the function $F^-,F,F^+$ respectively and write $t_0:=\kappa^*(0)$ as their common starting point. Note that the requirement for the curvature functions to be defined is equivalent to only integrating these graph functions on the domain where they are defined, as by lemma \ref{sigDEQ}, this gives us their inverses. Explicitly,
\begin{gather*}
    (\kappa^*)^{-1}(s) = \int_{t_0}^t \frac{dt}{F(s)},\\
    (\kappa^+)^{-1}(s) = \int_{t_0}^t \frac{dt}{F^+(s)},\\
    (\kappa^-)^{-1}(s) = \int_{t_0}^t \frac{dt}{F^-(s)}.
\end{gather*}
By construction, we also have that $F^-,F^+$ are the top and bottom curves of $IT(S,\delta),$ thus
\begin{gather*}
F^-\leq F\leq F^+\\
\Longrightarrow (\kappa^+)^{-1}(s) \leq (\kappa^*)^{-1}(s)\leq (\kappa^-)^{-1}(s).
\end{gather*}
Because each of these curvatures are monotonically increasing, this gives us the desired inequality. We make the above restriction on $s$ as again, $\kappa$ is only defined on $[0,L]$, and thus it is necessary that
$$
s\leq \rho^+(s) \leq L.
$$
\end{proof}

We now want to use the functions $\rho^\pm$ to bound how far apart two curvatures can be given they are in the same tube neighborhood. Because we are able to bound these curvatures as functions of $\kappa$, we will be able to produce explicit numerical bounds which we will cover in more detail in the following section.

\begin{lemma}\label{Ltau}
Let $0<\delta<\min\{|\kappa(0)-\kappa(L)|/2, m\}$, and define
\begin{gather*}
    \ell_\tau := \frac{\kappa^{-1}(\kappa(0) +\delta)}{1+\delta/m},\\
    L_\tau := \frac{\kappa^{-1}(\kappa(L) - \delta)}{1+\delta/m} - \ell_\tau.
\end{gather*}
Then, for all $s\in [0,L_\tau],$
\begin{gather*}
s(1-\delta/m) \leq \rho^-(s)\\
\leq \rho^+(s+\rho_\delta) \leq (s+\ell_\tau)(1+\delta/m)
\end{gather*}
\end{lemma}

\begin{proof}
Note that the condition on $\delta$ is necessary so that $\rho^\pm$ are defined and that $L_\tau>0.$
\\
Notice that by construction,
\begin{gather*}
0(1-\delta/m) =0= \rho^-(s)\\
< \rho^+(\rho_\delta) =\kappa^{-1}(\kappa(0)+\delta) = (\ell_\tau)(1+\delta/m),
\end{gather*}
and by lemma \ref{rho}, 
\begin{gather*}
    \frac{d}{ds}s(1-\delta/m) = 1-\delta/m \leq \frac{d}{ds}\rho^-(s)\\
    \leq \frac{d}{ds}\rho^+(s) \leq 1+\delta/m = \frac{d}{ds}(s+\ell_\tau)(1+\delta/m)
\end{gather*}
and so the result follows
\end{proof}

\begin{lemma} \label{tauBeta}
Let $S=\{(\kappa(s),\kappa'(s))\}$ be a phase portrait such that $\kappa'(s)\neq 0$. Take $0<\delta< \min\{|\kappa(0) - \kappa(L)|/2, m\}$. Let $S^*:=\{\kappa^*(s),(\kappa^*)'(s))\}\subset IT(S,\delta)$, and $d(S,S^*)<\delta$. Then for
\begin{gather*}
\tau(s) :=\kappa((s+ \ell_\tau)(1+\delta/m)),\quad s\in [0,L_\tau],\\
\beta(s):= \kappa(s(1-\delta/m)),\quad s\in [0,L_\tau],
\end{gather*}
for all $s\in [0,L_\tau]$,
$$
\beta(s) \leq \kappa(s), \kappa^*(s) \leq \tau(s)
$$
\end{lemma}

\includegraphics[width=\columnwidth]{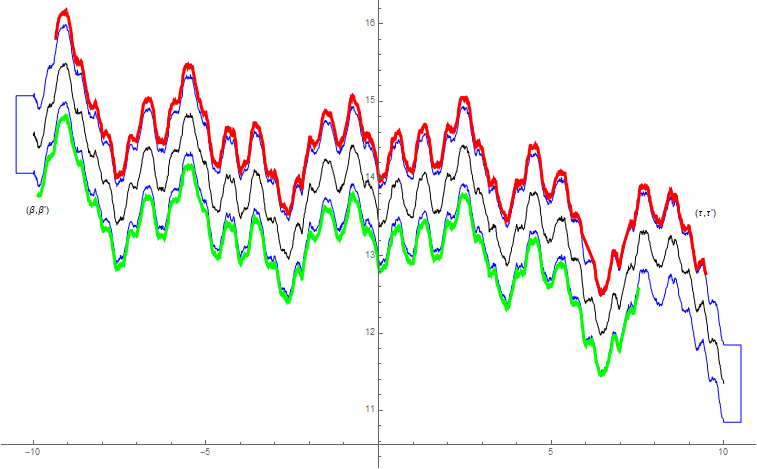}
\captionof{figure}{$(\tau,\tau')$ (red), and $(\beta,\beta')$ (green), plotted over the domain $[0,L_\tau]$. Observe that they are slightly above/below $\sigma^+,\sigma^-$ as they are slightly faster/slower.}

\begin{proof}
We assume our signature is in the upper half-plane. By lemma \ref{Ltau} we have that
{\small
\begin{gather*}
    s(1-\delta/m) \leq \rho^-(s) \leq \rho^+(s) \leq (s+\ell_\tau)(1+\delta/m),
\end{gather*}
}
so by the monotonicity of $\kappa$, we have
\begin{gather*}
    \beta(s) \leq \kappa^-(s) \leq \kappa^+(s) \leq \tau(s).
\end{gather*}
Then by lemma \ref{BottomTop}, because
$$
\kappa^-(s)\leq \kappa^*(s) \leq \kappa^+(s)
$$
the result follows. Observe that we must restrict to the interval $[0,L_\tau]$ as, by $d(S,S^*)<\delta$, we are only guaranteed that $\kappa(L) - \delta \leq \kappa^*(L^*)$, that is, $\kappa^*$ may only be defined on an interval up to where it achieves this value. Thus we choose $L_\tau$ so that
$$
\kappa*(s) \leq \kappa^*(L_\tau)\leq \tau(L_\tau) =\kappa(L)-\delta.
$$
\end{proof}

We now prove a result that specifies the same conditions as in lemma \ref{uncommonDomain}. Again, this is a slight strengthening of proposition 2, but one that will be much more convenient.
\begin{lemma}\label{generallyUseful}
Let $f'\neq 0$ and $S:=\{(f(s),f'(s));\ s\in [0,L]\}$. For all $\varepsilon>0$, there exists $\delta>0$ such that whenever another phase portrait 
 $S^*:=\{(g(s),g'(s));\ s\in[0,L^*]\}$ has the property $d(S,S^*)<\delta,\ S^*\subset T(S,\delta)$, then the conditions of lemma \ref{uncommonDomain} are met up to the bound $\varepsilon$. That is, there exists intervals $I=[x_1,y_1]\subset [0,L]$, $I^*=[x_2,y_2]\subset [0,L^*]$, where they have common length $y_1-x_1=y_2-x_2=\ell$, $\max\{d(I,[0,L]),d(I^*,[0,L^*])\}<\varepsilon$, and for all $s\in [0,\ell]$, $|f'(s+x_1) - g'(s+x_2)|<\varepsilon.$
\end{lemma}

\begin{proof}
We first specify that $\delta$ meets the requirements of the preceding lemmas. Without loss of generality we assume that $S$ is in the upper half plane.\\
Our main focus will be on signatures where $\kappa^*(0)\geq \kappa(0)$. This is because if it starts further left, that portion can only contribute $\delta/m_1$ length to the curve at maximum. We argue this formally.\\
Suppose that $\kappa^*(0)\leq \kappa(0)$. Then we can take $t_0$ so that $\kappa^*(t_0)=\kappa(0)$. Now, on $[0,t_0]$, $(\kappa^*)'(s)\geq m_1$, thus
$$
\kappa^*(\delta/m_1)\geq \kappa^*(0)+\delta \geq \kappa(0)
$$
so $t_0\leq \delta/m_1.$ Thus the parameter $x_2:=\delta/m_1$ can be made as small as desired and this portion of the domain can be disregarded as specified in the lemma.
\\
From lemma \ref{tauBeta}, we see that, for $s\in [0,L_\tau]$, for any signature $S^*$ as hypothesized,
we have that
$$
\beta(s) \leq \kappa^*(s), \kappa(s) \leq \tau(s),
$$
thus, on this interval,
$$
|\kappa^*(s) - \kappa(s)| \leq \max\{|\kappa(s) - \beta(s)|,|\kappa(s) - \tau(s)|\}.
$$
We could also take the distance between $\beta$ and $\tau$, but this gives a somewhat sharper bound. Importantly, $\kappa$ is $C^1$ on a compact domain, so it has Lipschitz constant $M:= \max_{s\in[0,L]}\kappa'(s)$. Thus, {\small $\forall s\in [0,L_\tau]$,
\begin{align*}
|\kappa(s) - \tau(s)| &= |\kappa(s) - \kappa((s+\ell_\tau)(1+\delta/m))|\\
& \leq M|s - s -s\delta/m - \ell_\tau(1+\delta/m)|\\
&= M(s\delta/m + \ell_\tau(1+\delta/m))\\
&\leq M(\delta L_\tau/m + \ell_\tau(1+\delta/m))
\intertext{as $L_\tau < L$,}
&< M(\delta L/m + \ell_\tau(1+\delta/m))\\
&=:\alpha_1(\delta).
\end{align*}
$\forall s\in [0,L_\tau]$ 

\begin{align*}
    |\kappa(s)-\beta(s)| &= |\kappa(s) - \kappa(s(1-\delta/m))|\\
    &\leq M|t\delta/m|\\
    &\leq ML\delta/m =:\alpha_2(\delta).
\end{align*}
}

where we write each $\alpha_i(\delta)$ functionally as the only non-constant involved is $\delta$. From the above, it is clear that $\lim_{\delta\to 0} \alpha_1(\delta)  = \lim_{\delta\to 0} \alpha_2(\delta)= 0$, upon recollection that 
$$
\ell_\tau:= \frac{\kappa^{-1}(\kappa(0) +\delta)}{1+\delta/m}.
$$ 
Thus,
{\small
$$
|\kappa(s) - \kappa^*(s)| \leq \alpha_3(\delta):= \max\{\alpha_1(\delta),\alpha_2(\delta)\} <\varepsilon,
$$
}
on $[0,L_\tau]$ for $\delta$ sufficiently small. Thus for $\ell$ the length of the common interval as stated in the lemma, we set $\ell=L_\tau$. What remains to be shown is that $|L_\tau-L^*|<\varepsilon$ and $|L-L_\tau|<\varepsilon$ for sufficiently small $\delta.$ The second of these inequalities is clear by definition of $L_\tau$, so we must only show the first.\\

What we want is to find a maximal domain for $\kappa^*$ so that we can bound $L^*$. To do this, we extend the function $\beta$ to be
{
$$
\beta(s):=
\begin{cases}
\kappa(s(1-\delta/m_1)), & s\in [0, \ell_\beta],\\
m_2s + k(L), & s\in [\ell_\beta,L_\beta],
\end{cases}
$$
}
where $\ell_\beta$ is the value such that
$$
\kappa(\ell_\beta(1-\delta/m_1)) = \kappa(L),
$$
and $L_\beta$ is such that $\beta(L_\beta) = \kappa(L) + \delta.$
It can be checked analogously to lemmas \ref{tauBeta} and \ref{BottomTop} that 
$$
\beta(s) \leq \kappa^*(s),
$$
thus, because $\kappa^*\leq \kappa(L)+\delta$ for all $s$, it must be that $L^*\leq L_\beta$ as
$$
\beta(L_\beta) = \kappa(L)+\delta.
$$
By construction, we see that 
$$
\ell_\beta = L/(1-\delta/m_1),\ L_\beta = \ell_\beta + \delta/m_2,
$$
and so $\lim_{\delta\to 0} L_\tau = \lim_{\delta\to 0} L_\beta = L$ as desired.
\end{proof}
The main theorem follows immediately from this lemma and lemma \ref{uncommonDomain}, as was previously remarked.
\newtheorem{corollary}{Corollary}

\subsubsection{Explicit Bound}
\begin{corollary} \label{ExpBound}
{
Let $S_\Gamma$ be a signature defined by the curve $\Gamma$, where $\Gamma$ has no vertices and has arc length parameterization $\gamma:[0,L]\to \mathbb{R}^2.$ Let $0<\delta<\max\{|\kappa(L)-\kappa(0)|/2,m\}$, and $\Gamma^*$ have signature $S_{\Gamma^*}$. If $d(S_{\Gamma^*}, S_\Gamma)<\delta$ and $S_{\Gamma^*}\subset T(S_\Gamma,\delta)$, then there exists $g\in SE(2)$ such that $d(\Gamma,g\Gamma^*)<\varepsilon(\delta)$ where
{
\begin{gather*}
    \varepsilon(\delta):= \max\{\delta/m_1, |L_\tau - L_\beta|\\
    + L_\tau^2 M(\delta L_\tau/m + \ell_\tau(1+\delta/m))/2\}
\end{gather*}
}
which is $O(\delta)$.
}
\end{corollary}

\includegraphics[width=\columnwidth]{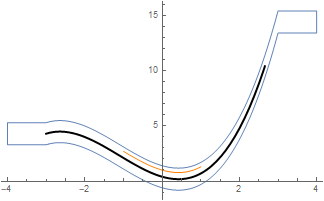}
\captionof{figure}{An example of two curves where one (orange) is contained in $T(S,\delta)$, but the Hausdorff distance between the two curves is greater than $\delta$.}

\medskip

\begin{proof}
We argue in two parts, fist parallel to lemma \ref{generallyUseful}, then lemma \ref{uncommonDomain}.
\\
In lemma \ref{generallyUseful}, we showed that there exists intervals $I=[x_1,y_1]\subset [0,L]$ and $I^*=[x_2,y_2]\subset [0,L^*]$ such that $y_1-x_1=y_2-x_2=L_\tau$ and for all $s\in [0,L_\tau],$
$$
|\kappa(s+x_1) - \kappa^*(s+x_2)|\leq \alpha_3(\delta).
$$
We note that $\alpha_3(\delta)=\alpha_1(\delta)$ as $\alpha_1(\delta)\geq \alpha_2(\delta).$
Further, we also have that
\begin{gather*}
    d(I,[0,L]) \leq |L_\tau - L|,\\
    d(I^*, [0,L^*]) \leq \max\{\delta/m_1, |L^* - L_\tau|\} \\
    \leq \max\{\delta/m_1, |L_\beta - L_\tau|\}.
\end{gather*}
If we now follow the proof of lemma \ref{uncommonDomain}, we see that we can choose $g\in SE(2)$ so that
\begin{gather*}
d(\Gamma. g\Gamma^*) \leq \max\{\max\{\delta/m_1,|L^* - L_\tau|\},\\
|L_\tau -L|\} +  L_\tau^2 M(\delta L_\tau/m + \ell_\tau(1+\delta/m))/2\},\\
\leq \max\{\delta/m_1,|L_\beta - L_\tau|\}\\
+  L_\tau^2 M(\delta L_\tau/m + \ell_\tau(1+\delta/m))/2.
\end{gather*}
As $L,L^*\leq L_\beta$. We can in fact do slightly better as the $g$ chosen is such that $\gamma(x_1) = \gamma^*(x_2)$, and so again, following the prior proof, we get,
\begin{gather*}
d(\Gamma. g\Gamma^*)\leq \max\{\delta/m_1,|L_\beta - L_\tau|
\\
+  L_\tau^2 M(\delta L_\tau/m + \ell_\tau(1+\delta/m))/2\}.
\end{gather*}
What is not immediately clear is what the order of this formula is. We can find an upper-bound for $\ell_\tau$ as we already have $m$ to be the minimum of $\kappa'$. Thus, to get from $\kappa(0)$ to $\kappa(0)+\delta$, it will be at most $\delta/m$ time, so we can replace $\ell_\tau$ with $\delta/(m(1+\delta/m))$. Similarly, we can replace $L_\tau$ with $(L-\delta/m)/(1+\delta/m) - \delta/(m(1+\delta/m))$ to completely remove inverses from the expression. $L_\beta$ does not involve inverses, so it does not provide any additional challenge computationally. With these substitutions, we see that our bound is $O(\delta)$.

\end{proof}
\newtheorem{remark}{Remark}
\begin{remark}
Note that for practical purposes, it may be best to choose the common point/tangent line not at the beginning of the curves $\Gamma,\Gamma^*$, but somewhere in their centers. The methodology above then would apply to considering the two halves of these curves as their deviations propagate moving away from this point.
\end{remark}

\includegraphics[width=\columnwidth]{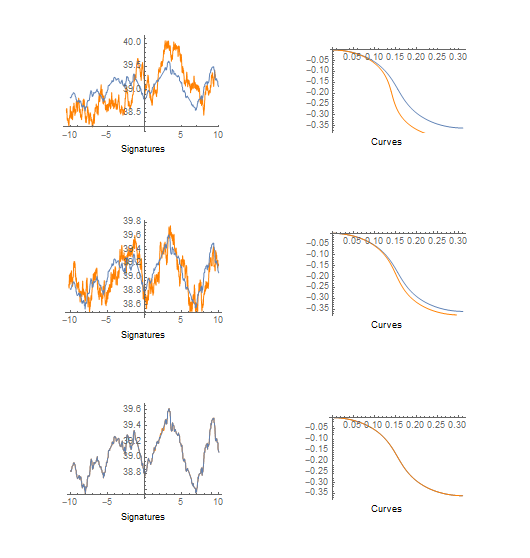}
\captionof{figure}{Here we take an initial blue signature, and take progressively closer orange signatures. As they get closer, we see that the defining curves also become closer.}

\medskip

\subsection{Interpreting Signatures through the $L^1$ metric}

It is clear that the argument above is sufficient but not necessary for two curves to be close together.

Small, sharp perturbations of a curve can cause large differences between signatures that are short-lived. As this example alludes to, perhaps an alternate approach to quantifying the distance between signatures could be more useful in this case. Intuitively, the area between two signatures seems like it should communicate the difference between curvatures, as the curvature is, in a sense, the accumulation of the $v$-coordinate. We begin by recalling lemma \ref{sigDEQ} which gives us a relationship between the signature as a graph of a function and as a phase portrait.

\includegraphics[width=\columnwidth]{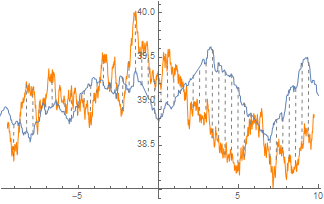}
\captionof{figure}{The area between two signatures.}

\medskip

\begin{lemma}
Let $F:[x,y]\to [m,M], F^*:[x,y]\to [m,M]$ be continuous, non-zero functions. For all $\varepsilon>0$, there exists $\delta>0$ so that, taking $u:[0,L]\to \mathbb{R},u^*:[0,L^*]\to \mathbb{R}$ such that
$$
(u(s), F(u(s))),\quad (u^*(s), F(u^*(s)))
$$ 
are in phase, 
$$
\int_x^y |F-F^*| < \delta \Longrightarrow |u(s) - u^*(s)|<\varepsilon
$$
for all $s\in[0,L-\varepsilon].$
\end{lemma}
\begin{proof}
Assume without loss of generality that $0<m<M$. By lemma \ref{sigDEQ}, we see that, for all $s\in[0,L],$
\begin{gather*}
|u^{-1}(s) - (u^*)^{-1}(s)| = |\int_x^t \frac{1}{F} - \frac{1}{F^*}|\\
\leq \int_x^t |\frac{1}{F} - \frac{1}{F^*}| \leq \int_x^t m^{-2}|F-F^*| \leq \frac{\delta}{m^2},
\end{gather*}
as $1/x$ has Lipschitz constant $m^{-2}.$ In other words, if $s=u^{-1}(s)$, then for $s^*=(u^*)^{-1}(s)$, $s-\delta/m^2 < s^* < s+\delta/m^2$, so that $u^*$ reached the value $t$ in at most $\delta/m^2$ ``time" before $u$, or $\delta/m^2$ after $u$. As $F,F^*$ are the derivatives of $u,u^*$, we see that $\max\{|u'|,|(u^*)'|\} \leq M$, thus
$$
|u(s) - u^*(s)|\leq \delta \frac{M}{m^2} <\varepsilon,
$$
for $\delta<\varepsilon m^2/M.$ The claim in regard to domains comes from the inequality
$$
|u^{-1}(y) - (u^*)^{-1}(y)|<\varepsilon,
$$
so that $|L^*-L|<\varepsilon.$
\end{proof}

\begin{corollary}\label{areaDist}
Let $F:[x,y]\to [m,M], F^*:[x,y]\to [m,M]$, $0<|m|<|M|$, be continuous, non-zero functions. Then
$$
\int_x^y |F-F^*| < \delta \Longrightarrow |u(s) - u^*(s)|\leq \delta M/m^2
$$
for all $s\in[0,L-\varepsilon].$
\end{corollary}

From this, the previous results can be worked out in a similar way to extend to curves that are not exactly over each other. This alternate viewpoint can be very useful when approaching more noisy signatures that might have smaller, more dramatic deviations in the estimated signatures. 
\begin{corollary}
Let $F:[x_1,y_1]\to [m,M], F^*:[x_2,y_2]\to [m,M]$, $0<|m|<|M|$, be continuous, non-zero functions, $\max\{x_1,x_2\}<\min\{y_1,y_2\}$. Let $\sigma(s)=(u(s), F(u(s)))$, $\sigma^*$ defined similarly be the graphs of two euclidean signatures. Then there exists $ g\in SE(2)$ such that
{\small
\begin{gather*}
\int_{\max\{x_1,x_2\}}^{\min\{y_1,y_2\}} |F-F^*| < \delta \Longrightarrow d(\Gamma, g\Gamma^*) \leq\\
\max\{|x_1-x_2|/m, |y_1-y_2|/m \\
+\delta M |\max\{x_1,x_2\} - \min\{y_1,y_2\}|/2m^3 + M\delta/m^2\}.
\end{gather*}}
\end{corollary}
\begin{proof}
Let $x:=\max \{x_1, x_2\}, y:=\min\{y_1,y_2\}$. We want to choose our g to align the points $\gamma(x),\gamma^*(x)$. Thus, parallel to lemma \ref{uncommonDomain}, the portions of the signatures between $x_1,x_2$ and $y_1,y_2$ can contribute at most $|x_1-x_2|/m$ and $|y_1-y_2|/m$ additional arc length respectively, and we can account for them as we did in corollary \ref{ExpBound}.\\
By corollary \ref{areaDist}, on the common domain of the in phase parameterizations of these signatures, $|u(s) - u^*(s)|<\delta M/m^2$. This common domain has maximal length $|\max\{x_1,x_2\} - \min\{y_1,y_2\}|/m$, so we can combine this with lemma \ref{closeCurva} to get the factor of $\delta M |\max\{x_1,x_2\} - \min\{y_1,y_2\}|/2m^3$ added to our bound.\\
Our final consideration is that the common domain of the two in phase parameterizations may not be the entire interval $[\max\{x_1,x_2\}, \min\{y_1,y_2\}]$, indeed, we only know that 
{\small
$$
|u^{-1}(\min\{y_1,y_2\}) - (u^*)^{-1}(\min\{y_1,y_2\})|<\delta /m^2.
$$
}
so there can be an additional length of $M\delta /m^2$ added to the error, as argued before.
\end{proof}
Observe that this formula is on the order of $1/m^3$, while the previous only depends on $1/m^2$, so while it can account for these small, large deviations, it may be weaker for signatures closer to the $u$-axis.

\subsection{Affine Signatures}

 Shifting focus, we can argue in parallel to the above for invariants under the Affine group $GL_n(\mathbb{R}) \ltimes \mathbb{R}^2.$ Here, there is a different notion of arc length and curvature which are invariant under this alternate group action. The affine arc length, $\alpha$, is given by the formula
 $$
 \alpha(s) = \int_0^t \sqrt{\kappa(s)}dt
 $$
 and likewise the affine curvature is given by
 $$
 \frac{1}{\kappa^{3/2}}\frac{d}{d\alpha}.
 $$
 As in the Euclidean case, this allows us to define the Affine signature to be the set $S:=\{\mu(\alpha), \mu_\alpha(\alpha)\}$. These functions are related to the Affine frame, $A(\alpha),$
 $$
 A(\alpha) = 
 \begin{bmatrix}
 T(\alpha)\\
 N(\alpha)
 \end{bmatrix}
 $$
 where $T,N$ are the tangent and normal vectors of the curve at $\gamma(\alpha)$ with derivative being taken with respect to the affine parameterization. Taking the Cartan matrix of $A(\sigma)$ to be
$$
C(A(\sigma)) = K(\sigma) = A'(\sigma)A^{-1}(\sigma),
$$
we get
$$
K(\alpha) = 
\begin{bmatrix}
0 & 1\\
\mu(\alpha) & 0
\end{bmatrix}
$$
where $\mu$ is the affine curvature. We further define $|A|:=\max_{i,j\in [n]}|a_{i,j}|$, which is the sup-norm metric, and observe that for all $\varepsilon>0$, there exists $\delta>0$ so that
$$
|\mu(\alpha) - \mu^*(\alpha)|<\delta \Longrightarrow |K(\sigma) - K^*(\sigma)| < \varepsilon,
$$
that is, if two curves have close affine curvature with respect to affine arc length then these cartan matrices will be close. In particular, it suffices to take $\delta = \varepsilon.$ Another useful observation will be that for $K\in gl_n(\mathbb{C})$ if $|K|<M$, then
$$
|K^m|<(Mn)^m.
$$
It has been established how we can reconstruct a curve from its euclidean curvature, but the following lemma from \cite{Guggenheimer} shows us how this can be done  more generally via Picard Iterations.
\begin{lemma}\label{PicardIterations}
Let $K(s)$ be a continuous matrix function, $s_0\leq s \leq s_1$. For any $s^*,$ $s_0\leq s^* \leq s_1$, there exists an interval about $s^*$ and a nonsingular matrix function $A(s)$ definedin the interval such that $K(s) = C(A(s)),$ $A(s^*) = U$ where $U$ is a generic nonsingular matrix.
\end{lemma}
\begin{proof}
We observe that
$$
A'A^{-1} = K \Longleftrightarrow A' = KA,
$$
so it suffices to solve the latter with the given initial condition. We iteratively define
$$
A_0 = U,\quad A_j(s) = U + \int_{s^*}^s K(\sigma)A_{j-1}(\sigma) d\sigma
$$
where the integration is element-wise. $K$ is continuous on $[s_0, s_1]$, so it achieves a maximum $M:=\max_{s\in [s_0,s_1]} |K(s)|$. We now argue by induction that
$$
|A_j(s) - A_{j-1}(s)| \leq n^{j-1}M^j\frac{|s-s^*|^j}{j!}.
$$
For our base case, we have that
$$
|A_1(s) - A_0| \leq \int_{s^*}^s|K(\sigma)| |d\sigma| \leq M|s-s^*|.
$$
Assuming the inductive hypothesis, we get
{\footnotesize
\begin{align*}
    |A_{j+1}(s) - A_j(s)| &= |\int_{s^*}^sK(\sigma)(A_j(\sigma) - A_{j-1}(\sigma))d\sigma |\\
    &\leq nMn^{j-1}M^j \int_{s^*}^s \frac{|\sigma - s^*|^n}{n!}\\
    &= n^j M^{j+1}\frac{|s-s^*|^{n+1}}{(n+1)!}.
\end{align*}
}
As this is the $(n+1)$th term in the taylor expansion of $e^{nM|s-s^*|}/n$, it tends to zero with increasing $j$, as does $|A_i - A_j|$ for $i>j$. Therefore, $A(s)=\lim_{j\to\infty}A_j(s)$ is our solution, and satisfies the equation
$$
A(s) = U + \int_{s^*}^s K(\sigma) A(\sigma) d\sigma.
$$
We also observe that this convergence is uniform as $|s-s^*| \leq |s_0 - s_1|.$
\end{proof}

Before proceeding any further, we will now outline the main steps in our argument. For any $\varepsilon>0$ and affine signatures $S,S^*$, we want to find $\delta_j>0$ and $g$ in the affine group so that,

\begin{gather}
    d(S, S^*) < \delta_1 \Longrightarrow |\mu(\alpha) - \mu^*(\alpha)| <\delta_2\\
    \Longrightarrow |K(\alpha) - K^*(\alpha)|<\delta_3\\
    \Longrightarrow |A(\alpha) - A^*(\alpha)| <\delta_4\\
    \Longrightarrow d(\Gamma, g\Gamma^*) < \varepsilon
\end{gather}
where the first three inequalities are only valid on the common domain of $\mu,\mu^*$, and the remaining domain is arbitrarily small. Inequality $(1)$ was already shown generally for phase portraits that do not intersect the $u$-axis. $(2)$ was already discussed, so what remains is to show (3) and (4).\\
We start now with acquiring the necessary tools to go from (2) to (3). What we want to do first is bound the rate of convergence of the previously discussed Picard Iterations. From now on, we will use the notation $K_j$ rather than $A_j$ to clearly indicate to which matrix the picard iterations correspond. We prove the following for general matrix functions.
\begin{lemma} \label{uniformBound1}
Let $|M(\sigma)|$ be any continuous matrix function on $[0,L]$. Then for all $\varepsilon>0$, there exists $N>0$ where for all $K(\sigma)$ on this interval, $|K(\sigma)|\leq |M(\sigma)|$, if $j>N$, $|K_j(\sigma) - A_K(\sigma)|< \varepsilon$.
\end{lemma}
\begin{proof}
Let $M$ be the max of $M(\sigma)$, so that for any such $K(\sigma)$, $|K(\sigma)|\leq M$. Then by lemma \ref{PicardIterations}, we see that
$$
|K_j(s) - K_{j-1}(s)| \leq n^{j-1}M^j\frac{|s-s^*|^j}{j!}=:\alpha_j(s).
$$
We recall that this corresponds to the $(n+1)$th term of the taylor expansion of $e^{nM|s-s^*|}/n$, so the sum of these terms converges, and for any $\varepsilon>0$, there exists $N>0$ so that for $j_1,j_2\geq N$,
$$
|K_{j_1}(s) - K_{j_2}(s)| \leq \sum_{i=j_1}^{j_2} \alpha_i(s) <\varepsilon/2.
$$
In particular, this tells us that the sequence\\ $(K_j)_{j=N}^{\infty} \subset (K_N)_{\varepsilon/2}$, and so $|A_K - K_j|\leq \varepsilon/2 <\varepsilon$, as desired.
\end{proof}
\begin{lemma} \label{uniformBound2}
Let $|M(\sigma)|$ be any continuous matrix function on a compact interval. Then there exists $M>0$ such that for any $|K(\sigma)|\leq |M(\sigma)|$ defined on a domain $[0,L]$, for all $j$, $K_j(\sigma) \leq M$.
\end{lemma}
\begin{proof}
By lemma \ref{PicardIterations}, we have that, taking $M_1\geq |M(\sigma)|$,
$$
|K_j(\sigma) - K(\sigma)| \leq \sum_{i=1}^j \alpha_i(s) \leq e^{nM_1L}/n.
$$
Thus
$$
|K_j(\sigma)| \leq M_1 + e^{nM_1L}/n =:M
$$
as desired.
\end{proof}
\begin{lemma}
Let $K(\sigma)$ be a continuous matrix function. For all $\varepsilon>0$, there exists $\delta>0$ so that $|K^*(\sigma)-K(\sigma)|<\delta$ implies $|A_K(\sigma) - A_{K^*}(\sigma)|<\varepsilon.$ 
\end{lemma}
\begin{proof}
We first show that for all $\varepsilon>0$, all $N>0$, there exists $\delta>0$ so that $|K_j(\sigma)- K_j^*(\sigma)|<\varepsilon$ given the same initial condition $U$. We verify this by induction. The base case is trivial as 
$$
|K_0 - K^*_0| = |U - U| = 0.
$$
Assume the hypothesis for $j$, and take $\delta<1$. Then by lemma \ref{uniformBound2}, we can take $M>|K|,|K^*_j|$ so that
\begin{gather*}
   | K_{j+1}(s) - K_{j+1}^*(s)| \\
   = |\int_{s^*}^s K(\sigma) K_j(\sigma) - K^*(\sigma)K_j^*(\sigma)d\sigma|\\
   = |\int_{s^*}^s K(\sigma) K_j(\sigma) - K(\sigma) K_j^*(\sigma)\\
   + K(\sigma) K_j^*(\sigma) - K^*(\sigma)K_j^*(\sigma)d\sigma|\\
   \leq \int_{s^*}^s | K(\sigma)(K_j(\sigma) - K_j^*(\sigma))|d\sigma\\
   + \int_{s^*}^s | K_j^*(\sigma)(K(\sigma) - K^*(\sigma))|d\sigma\\
   \leq |s-s^*| 2M n \varepsilon < L 2M n \varepsilon
\end{gather*}
As this can be made arbitrarily small, we have shown the first claim. By the initial choice of $\delta$, we can further restrict it, and apply lemma \ref{uniformBound1} and the above to get $N>0$ so that
\begin{gather*}
    |A_K(s) - A_{K^*}(s)|\\
    = |A_K(s) - K_N(s) + K_N(s) - K_N^*(s)\\
    + K_N^*(s) - A_{K^*}(s)|\\
    \leq |A_K(s) - K_N(s)| + | K_N(s) - K_N^*(s)|\\
    + |K_N^*(s) - A_{K^*}(s)| < 3\varepsilon,
\end{gather*}
as desired. 
\end{proof}
We now have the desired relationship between $A_K$ and $A_{K^*}$ on their common domain, and because we will always be taking curvatures from close signatures, the remaining domain will be small (lemma \ref{generallyUseful}). In the following, we show what happens on the common domain, and then approach the remaining domain.

\begin{lemma}
For all $\varepsilon>0$, there exists $\delta>0$ $|A_K(\alpha)-A_{K^*}(\alpha)|<\delta$ for all $\alpha\in [0,L]$, then $d(\Gamma, g\Gamma)<\varepsilon$ for some $g$ in the affine group.
\end{lemma}
\begin{proof}
We again choose $g$ so that the initial points and tangents agree. Then we see that
\begin{gather*}
    |\gamma(\alpha) - g\gamma^*(\alpha)| = |\int_0^\alpha T(s) - T^*(s)ds|\\
    \leq \int_0^\alpha \delta ds = \alpha \delta \leq L\delta < \varepsilon
\end{gather*}
for $\delta<\varepsilon/L.$
\end{proof}

\begin{lemma}
Let $A_K$ be the affine frame for $\gamma$, and let $M>0$ be such that $|A_K|\leq M$. Then $|\gamma(a)-\gamma(b)|\leq (b-a)M.$
\end{lemma}

\begin{proof}
\begin{gather*}
    |\gamma(a) - \gamma(b)| = |\int_a^b T(s)ds| \leq (b-a)M.
\end{gather*}
\end{proof}
We see then that we have reproduced the necessary pieces to prove lemma \ref{uncommonDomain} in the affine case. Because lemma \ref{generallyUseful} also applies, we get the desired result.
\begin{theorem}
Let $S$ be an affine signature with no intersection with the $u$-axis. Then for all $\varepsilon>0$, there exists $\delta>0$ so that if $d(S,S^*)<\delta$, then $d(\Gamma,g\Gamma^*)<\varepsilon$ for some $g$ in the Affine group.
\end{theorem}

What we can observe is that this argument is not particularly special to the Affine group. In fact, the same argument will hold for any group where we can establish a frame with the following two properties.
\begin{itemize}
    \item There is a curvature associated with the cartan matrix of the frame that has a continuous relationship with all of its components.
    \item The frame has a continous relationship with the tangent of the curve $\gamma$.
\end{itemize}

\section{Closed Curves}

\subsection{The Non-degenerate Case}
We first show an extension of a result from \cite{Geiger} regarding the euclidean signature, where we take the euclidean curvature $\kappa$ and arc length parameterization $s$. Here we must additionally stipulate that our curves are non-degenerate, that is, they are such that $\kappa_s(s)=0$ for only finitely many values of $s$ in the minimal period of $\kappa$. We will use the notation $\sigma_\gamma(p):\mathbb{R}^2\to\mathbb{R}^2$ to refer to the map from a point on the curve to its corresponding point on the signature. We highlight the following argument as it stresses the point that our condition really ought to be on the signature and only implicitly on the curve itself, as in the proper conditions, the curve will be determined by its signature, and vice versa. We cite the fundamental result from \cite{Geiger}.
\begin{theorem}
\label{foundation}
Let $\Gamma_1$ and $\Gamma_2$ be two curves parameterized by $\gamma_1:\mathbb{R}\to\mathbb{R}^2$, and $\gamma_2:\mathbb{R}\to\mathbb{R}^2$ respectively. Suppose that there exists open intervals $I_1,I_2\subset \mathbb{R}$ where $\sigma_1|_{I_1}$ and $\sigma_2|_{I_2}$ are injective and have common image $S'$ homeomorphic to $\mathbb{R}$. Then $\Gamma_1':=\gamma_1(I_1)$ and $\Gamma_2':=\gamma_2(I_2)$ are congruent.
\end{theorem}

Later in this paper, we will show that we can drop the non-degeneracy condition entirely, however in the preceding argument that condition was necessary.
\\
As will be re-emphasized later, the same result holds if we take the $i$th order signature 
$$
S^{(i)} := \{\kappa(s), \frac{d}{ds}\kappa(s), \dots, \frac{d^i}{ds^i}\kappa(s)\}\subset \mathbb{R}^{i+1}
$$ 
in place of $S$. This slightly strengthens the result, as when we take higher order derivatives, it becomes increasingly likely that $S^{(i)}$ is injective. The importance of this idea is that higher order derivatives will also determine the curvature under the proper circumstances. The proof goes through in the same way, as it is only the first two coordinates that really matter, but what this gives us is an additional means to identify different curves. If two open curves have the same signature, but that signature intersects itself at say, points $s_1,s_2$ of its parameterization, the above criteria would not be able to tell if they are congruent or not. However, looking at this location on the curve, we see that it is of the form $(\kappa(s_1),\kappa_s(s_1)) = (\kappa(s_2), \kappa_s(s_2))$, so upon differentiating to a higher order the point of intersection may dissipate. It is only when the curvatures agree at all derivatives that this additional criteria again can fail. Intuitively, this tells us that as long at each point of the signature there is only one path that preserves smoothness of the curve, then there is a unique curve that it corresponds to.\\
While artificially, it is not impossible to construct curves that have this ambiguity, in the wild this is exceedingly unlikely, and the following argument is meant to express this rigorously.\\
Let $\kappa(s)$ be $C^2$ such that there exists finitely many points $\{s_n\}_{n=1}^N$ where $\sigma(s_i)=\sigma(s_j)$ for all $i,j$. Looking at just two of these points, we see that for their intersection to persist in $S^{(2)}$, it must be that $\kappa_{ss}(s_1) = \kappa_{ss}(s_2).$ Define $x:= \kappa_{ss}(s_1)$, $y:=\kappa_{ss}(s_2)$. These values are independent of each other, as smoothness is a local property, so that values at $s_1$ and $s_2$ in principal have no relation to each other. Thus they must independently fall in the set $\{(x,y)\ |\ x=y\} \subset \mathbb{R}^2$. As this is the graph of a continuous function, it has Lebesgue measure 0. As the Lebsgue measure is sub-additive, this argument then extends to the larger collection of intersections to show that the intersection persisting is probability 0.

\subsubsection{Closed Curves}

We will now use the theorem from the previous section to prove an analogous result for closed curves that are not necessarily simple (the proof for simple curves was done previously in \cite{Geiger}).
\begin{theorem}\label{closedCurveUnique}
Assume $\Gamma_1$ and  $\Gamma_2$ are closed non-degenerate curves with the same $i$th order signature $S^{(i)}$, which is simple and closed. Then $\Gamma_1$ and $\Gamma_2$ are congruent.
\end{theorem}
\begin{proof}
Let $\gamma_1$ and $\gamma_2$ be unit speed parameterizations of $\Gamma_1$ and $\Gamma_2$, respectively and let $p_1=\gamma_1(0)\in \Gamma_1$ and $q_0=\sigma_{\gamma_1}(p_1)\in S$. As our signatures have the same image, there exists $p_2\in \Gamma_2$ such that $\sigma_{\gamma_2}(p_2) = q_0$. In short, we ensure that our parameterizations start in the same place on the signature. Let $\gamma_1$ and $\gamma_2$ have minimum periods $L_1$ and $L_2$ respectively, and translate them so that their domains are $[0,L_1]$ and $[0,L_2]$ where we take $L_1,L_2$ to be their minimal period, and they both have initial points $p_1,p_2$ respectively. Define $t_0^1:= 0 < t_1^1 < \dots < t_{n_1}^1 :=L_1$, $t_0^2:=0 < t_1^2 < \dots < t_{n_2}^2:=L_2$, where $t_i^j$ corresponds to the $i$th point of self-intersection in the parameterization $\gamma_j$.\\
 Let us first consider the intervals $(t_0^1, t_1^1)$ and $(t_0^2, t_1^2)$. We argue $\sigma_{1}((t_0^1, t_1^1)) = \sigma_{2}((t_0^2, t_1^2))$. Suppose not. Both open arcs have a common starting point $\sigma_{1}(0) = \sigma_{2}(0)$ by construction, and because the traversal of the signature is given by the values of $\kappa, \kappa_s$ encoded in this point, both curves must be traversed in the same direction from here, thus it must be that one properly contains the other.\\
Without loss of generality, let $\sigma_{1}((t_0^1, t_1^1)) \subsetneq$\\ $\sigma_{2}((t_0^2, t_1^2))$. It follows that
$$
\sigma_{1}(t_1^1) = \sigma_{2}(t^*) \in \sigma_{2}((t_0^2, t_1^2)),
$$
however, by theorem \ref{foundation}, $\gamma_1((t_0^1, t_1^1)) = \gamma_2((t_0^1, t^*))$, but the closure of the first has a point of intersection, whereas in the second $t^*<t_1^1$, so the first point of intersection is not achieved. Thus we have derived a contradiction and $t_1^1 = t_1^2$.\\
We can repeat this argument inductively to verify that $\gamma_1([0, t_{i}^1]) = \gamma_2([0, t_{i}^2])$ for all $i$, so it follows that $n_1=n_2$, and the curves must be congruent.
\end{proof}

\section{Partitions of The Signature}

Our following methodology will be based on section 1. This gives us a very different perspective on the problem, as we will be arguing from the viewpoint of the signature as a differential equation $\sigma(s) = (\kappa(s), \kappa_s(s)),$ rather than through the use of topological arguments as above.
\\
As such, we re-introduce notation regarding phase-portraits, that is, sets of the form
$$
S^{(i)}:= \{(f,f', f'',\dots, f^{(i)})\}.
$$
We say that a parameterization
$$
\sigma^{(i)}(t) := (f(t),f_1(t),\dots, f_i(t))
$$
is in-phase if $f_i=f^{(i)}_t.$
\\
In the following, a signature or $i$th order signature will not necessarily refer to a euclidean signature, but rather the set as described before, where our curvature and parameterization are chosen with respect to the action of a particular group $G$. 
\\
We recall the following theorem from section 2, re-worded in terms of this new language.
\begin{theorem} \label{MainTheoremOpenCurves}
Let $S$ be a 1st order signature such that $\kappa_\alpha\neq 0$. Then for all $\varepsilon>0$, there exists $\delta>0$ so that if $d(S,S^*)<\delta$, then $d(\Gamma,g\Gamma^*)<\varepsilon$ for some $g\in G$. In fact, if $\gamma:[0,L_1]\to \mathbb{R}^2$, and $\gamma^*:[0,L_2]\to \mathbb{R}^2$, then we can choose such a $g,\delta$ so that setting $L:=\max\{L_1,L_2\}$, $\ell:= \min\{L_1,L_2\}$
$$
|\gamma(t) - g\gamma^*(t)| < \varepsilon,\quad \forall t\in [0,\ell]
$$
and for $t_1:= \min\{t,L_1\}$, $t_2:= \min\{t,L_2\}$,
$$
|\gamma(t_1) - \gamma(t_2)| < \varepsilon,\quad \forall t\in [\ell,L].
$$
\end{theorem}

What this says is that if our two signatures are close enough to each other, then we can expect the curves themselves to be traversed in more or less the same way. The focus of this section will be to extend this result to the case where some higher-order derivative of the curvature is non-zero. The main idea of this proof is that we will be able to locally verify that the curves are congruent due to the non-zero derivatives of the curvature and similar initial values, and then we will be able to piece together this local data to cover our curve and show global congruence. 
\\
We define a \textit{partition} of $S^{(i)}$ to be a sequence $0=t_0<t_1<\dots<t_n = L$ , $f^{(k_j)}(t_j)\neq 0$ for any $0<j<n$, that has the following property. Let $\sigma^{(i)}(t):=(f(t), f'(t), \dots, f^{(i)}(t))$, and define $S_j := \sigma^{(i)}([t_{j-1}, t_j])$ with in phase parameterization $\sigma_j:=(f_j,\dots, f_j^{(i)})$. For all $j\in [n]$, for all $\varepsilon>0$, there exists $\delta>0$ such that, for any $S_j^*$, $\sigma_j^*:[0,L^*]\to \mathbb{R}^i$ defined similarly, $\ell,L,t_1,t_2$ as before,
\begin{gather*}
d(S_j, S^*_j)<\delta_j \Longrightarrow\\
|\sigma_j(t) - \sigma_j^*(t)| < \varepsilon,\ \forall t\in [0, \ell],\\
|\sigma_j(t_1) - \sigma_j^*(t_2)| < \varepsilon,\ \forall t\in [\ell, L],
\end{gather*}
and $|L-L^*|<\varepsilon$.\\
{\center{
\includegraphics[width=0.45\columnwidth]{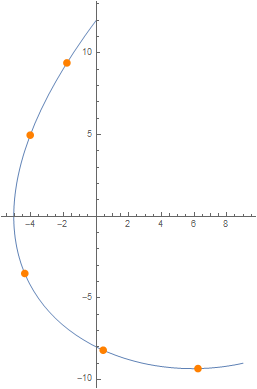}
\captionof{figure}{A partition depicted on a 1st order signature. In this case, the condition for one derivative to be non-zero is equivalent to the point not lying on the $x$-axis.}
}}

\medskip
Our main motivation for this definition, and the following lemmas, will be to establish that it is enough that our curve has no high order vertices, that is, points of the form $(x,0,\dots,0)\subset S^{(i)}.$\\
We see that the definition of a partition corresponds very nicely to the case elaborated on in theorem \ref{MainTheoremOpenCurves}, and it in essence spells out how exactly we will approach the proof. Having a partition essentially states that we know how to handle our phase portrait between each of these points $\sigma^{(i)}(t_j)$, and what we will want to show next is that we can extend this information to the phase portrait globally. From now on, we will be using $\kappa,\kappa^*$ rather than $f$ and $g$ as these align more with the notation used in arguments regarding curvatures, although in principle they are still only being viewed as phase portraits. Our argument will take multiple steps, but the main idea is that we want to break our $S_*^{(i)}$ into pieces $S_j^*$ that correspond as desired to each $S_j$.

\begin{lemma}\label{injectiveSig}
Let $0=t_0<t_1<\dots<t_n = L$ be a partition of $S^{(i)}$ an injective signature. Then there exists $\delta>0$ so that $(S_i)_{\delta} \cap (S_j)_{\delta} = \emptyset$ if $|i-j|>1$.
\end{lemma}

\begin{proof}
We first see that each $S_i$ is compact, and for $|i-j|>1$, $S_i\cap S_j=\emptyset$. Thus $d(S_i,S_j)>0$, and we can separate them by uniform neighborhoods. Because there are only finitely many such $S_j$, we can take the smallest necessary neighborhood, and repeat this procedure for each $S_i$.
\end{proof}

{\center{
\includegraphics[width=0.45\columnwidth]{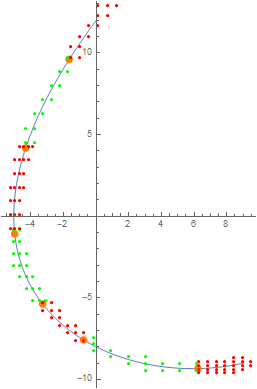}

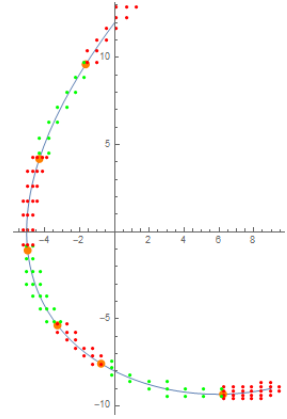
\captionof{figure}{A depiction of choosing uniform neighborhoods as specified in the above lemma.}
}}

\medskip

\begin{lemma}\label{gluing1}
Let $0=t_0<t_1<\dots<t_n = L$ be a partition of $S^{(i)}$ an injective signature, and let $S_*^{(i)}$ be another phase portrait. For any $\varepsilon>0$, there exists $\delta>0$  such that if $d(S^{(i)},S_*^{(i)})<\delta$ then there exists a path component $S_j^*\subset S_*^{(i)}\cap (S_j)_\delta$, with in phase parameterization $\sigma_j^*$ such that $|\sigma_j^*(0)-\sigma_j(0)|<\varepsilon$ and $|\sigma_j^*(L_j^*)-\sigma_j(L_j)|<\varepsilon$.
\end{lemma}
\begin{proof}
Take $\delta:=\min\{\delta_j\}$, and further restrict it so that $(S_i)_{\delta} \cap (S_j)_{\delta} = \emptyset$ if $|i-j|>1$ by lemma \ref{injectiveSig}.\\
Define $s_\ell:=\gamma_j(0),\ s_r:=\gamma_j(L)$, $W_1:= (S_{j-1})_\delta\cap (S_{j})_\delta,\ W_2:=(S_{j+1})_\delta\cap(S_{j})_\delta.$ Because\\ $d(S^{(i)},S_*^{(i)})<\delta$, we have that there is some $p_1\in S_*^{(i)}\cap W_1$ and $p_2\in S_*^{(i)}\cap W_2$. Now, $p_1$ and $p_2$ are connected by a path in $S_*^{(i)}$ as this set is a path. By our choice of $\delta$, because each of these neighborhoods is disjoint accept to the neighborhoods adjacent to them, this path must go through $(S_j)_\delta$. Thus, there is a path component $S_j^*\subset (S_j)_\delta$ with initial point in $W_1$ and terminal point in $W_2$ (because this is a phase portrait its orientation necessitates this), concluding our proof.
\end{proof}

\begin{lemma}\label{epball}
For all $\varepsilon>0$ we can choose $\delta>0$ so that $W_1\subset (s_\ell)_\varepsilon$ and $W_2 \subset (s_r)_\varepsilon.$
\end{lemma}
\begin{proof}
This is trivial for $W_1$ or $W_2$ respectively if $j=1,n$ as they coincide with the uniform neighborhood, so assume otherwise. We can also assume that the first derivatives are non-zero, and $\kappa_s(t_{j-1}),\kappa_s(t_j)>0$. This is because we will only need to apply our argument to the next highest derivative, so it is inconsequential which one we start at. We will only consider $W_2$, as the argument is mirrored.\\
Set $2\alpha := \kappa_s(t_j)$. As $\kappa_s$ is continuous, there exists some $\beta>0$ where for all $t\in [t_j-\beta, t_j+\beta]$, $\kappa_s(t)>\alpha$. In particular, $\kappa$ is increasing on this interval, and
$$
|\kappa(t_j+\beta) - \kappa(t_j -\beta)| \geq \beta \alpha.
$$
Thus if we take $\delta<\beta\alpha /2$, we see that $\kappa(t_j+\beta), \kappa(t_j -\beta)\not\in W_2$ as
$$
\kappa(t_j\pm \beta) \not\in (\kappa(t_j \mp \beta))_\delta,
$$
so that $W_2\subset (\sigma^{(i)}([t_j-\beta, t_j+\beta]))_\delta$. As $\sigma^{(i)}$ is continuous, for all $\varepsilon>0$, there exists $\beta>0$ such that 
\begin{gather*}
\sigma^{(i)}([t_j-\beta, t_j+\beta]) \subset (\sigma^{(i)}(t_j))_\varepsilon\\
\Longrightarrow (\sigma^{(i)}([t_j-\beta, t_j+\beta]))_\delta \subset (\sigma^{(i)}(t_j))_{\varepsilon+\delta}
\end{gather*}
and the claim follows.
\end{proof}
What this lemma tells us is that there are certain portions of $S_j^*$ that suit us well, however we have yet to verify that we can connect them seamlessly. We now show that this is resolved upon small enough choice of $\delta.$
\begin{lemma}\label{gluing2}
Let $0=t_0^1<t_1^1<\dots<t_n^1 = L$ be a partition of $S^{(i)}$and let $S_*^{(i)}$ be another phase portrait. For all $\varepsilon>0$, there exists $\delta>0$ so that given $d(S_*^{(i)}, S^{(i)})<\delta$, we can find a sequence $0=t_0^2<t_1^2<\dots<t_n^2=L^*$ where $|\kappa^*(t_j^2) - \kappa(t_j^1)|<\varepsilon,$
and $|t_j^1 - t_j^2| <\varepsilon.$
\end{lemma}

\begin{proof}
By lemma \ref{gluing1}, we can take $\min\{\delta_j\}>\delta>0$ so that we get segments $S_j^*\subset S_*^{(i)}\cap (S_j)_\delta$. Under these conditions, we can invoke the properties of the partition to get
\begin{gather*}
|\kappa_j(0)-\kappa_j^*(0)|<\varepsilon,\\ |\kappa_j(L_j)-\kappa_j^*(L_j^*)|<\varepsilon.
\end{gather*}
We now take $0=t_0^2<t_1^2<\dots<t_n^2=L^*$ so that $\sigma_*^{(i)}(t_j^2) = \sigma_*^j(0)$, that is, we take these to be the initial points of each $S_j^*$. We are guaranteed this same ordering as it follows the orientation of $S^{(i)}$. Define $\mathcal{L}_j^*:= t_{j}^2 - t_{j-1}^2$, and let $L_j^*$ be such that
$$
\sigma_*^j([0,L_j^*]) = S_j^*.
$$
What we mean here is that our choice of the path component $S_j^*$ gives us some initial value for how long its domain should be ($L_j^*)$. What we want to show next is that we can safely extend it so that we can include the filler between $S_j^*$ and $S_{j+1}^*$ while maintaining the properties of the partition. That is, at the end of the day we want it to be safe to replace $L_j^*$ with $\mathcal{L}_j^*$. We first want to show that for any $\varepsilon>0$, we can choose $\delta>0$ so that $|\mathcal{L}_j^* - L_j|<\varepsilon$. In the case where $L_j^*\leq \mathcal{L}_j^*\leq L_j$, we have that $|\mathcal{L}_j^*- L_j| \leq |L_j^* - L_j|<\varepsilon$ choosing $\delta$ as given in the partition hypotheses. Thus we cam assume $\mathcal{L}_j^*>L_j.$\\
By lemma \ref{epball}, we can take $\delta$ such that
$$
\sigma_*^{(i)}(\mathcal{L}_j^*),\sigma_*^{(i)}(L_j^*)\in W_2\subset (s_r)_\varepsilon,
$$
as this is where $S_j^*$ ends, and $S_{j+1}^*$ begins. We now argue that there must be a lower bound for ``how long" it will take before $\sigma_*^{(i)}$ to leave $(s_r)_\varepsilon$ and not return locally. This is because our derivative being non-zero will push the curve away from this point. We make this formal in the following.\\
We again assume that the first derivative is non-zero and positive without loss of generality. Recall again that $s_r = \sigma^{(i)}(t_j^1)$, and so $2\alpha:=\kappa_s(t_j^1)> 0$. By the continuity of $\kappa_s$, we can take $\beta>0$ so that for all $t\in[t_j^1-\beta, t_j^1+\beta]$, $\kappa_s(t) -\delta >\alpha$ if $\delta<\alpha/2$. Fix such a $\beta$, $\delta'$ and take $\varepsilon>0$, so that
$$
W_2\subset (s_r)_{\epsilon} \subset (\sigma^{(i)}([t_j^1-\beta, t_j^1+\beta]))_{\delta'}.
$$
In partiular, take $\varepsilon< |\sigma^{(i)}(t_j^1) - \sigma^{(i)}(t_j^1\pm \beta)|$. Combining this with the preceding discussion, we have for $|\kappa^*(t) - \kappa(t_j^1)|<\varepsilon$, $\kappa^*_s(t)>\alpha$. Once $\kappa^*(t)>\kappa(t_j^1) + \varepsilon$, because $\kappa^*$ can only increase locally, we will have 
$$
|\sigma_*^{(i)}(t) - \sigma^{(i)}(t_j^1)|\geq |\kappa^*(t) - \kappa(t_j^1)|>\varepsilon
$$
and so $\sigma_*^{(i)}(t)\not\in W_2$. By our choice of $\delta$, we recall that $\kappa_s^*(t)>\alpha$, so that 
$$
|\kappa^*(t + L_j^*) - \kappa^*(L_j^*)|\geq \alpha t > \varepsilon
$$
for  $t > \varepsilon/\alpha$. This implies that $|\mathcal{L}_j^*-L_j^*|< \varepsilon/\alpha$ which can be made arbitrarily small by choice of $\delta.$ In summary, we can choose $\delta>0$ so that $|L_j^* -L_j|<\varepsilon/2$, and $|L_j^* - \mathcal{L}_j^*|<\varepsilon/2$. Thus 
$$
|L_j - \mathcal{L}_j^*| < \varepsilon
$$
as desired.\\
The remaining results follow easily after this. By hypothesis we have that $$M:=\max|\kappa^*_s| \leq \max|\kappa_s| + \delta,$$ and so $|\kappa^*(L_j^*) - \kappa^*(t)|\leq M|L_j^* - t| \leq  M|L_j^* - \mathcal{L}_j^*|$ can be made less than $\varepsilon$ for all $t\in [L_j^*, \mathcal{L}_j^*].$
\end{proof}
From this, we immediately get the next result.
\begin{corollary}\label{PartitionMeat}
If we can find a partition of $S^{(i)}$, then for all $\varepsilon>0$, there exists $\delta>0$ so that, for $d(S^{(i)},S_*^{(i)})<\delta$, we can parameterize $\kappa^*$ in phase to have domain $[0,L^*]$, where $|L-L^*|<\varepsilon$, and for all $t\in[0,\ell:=\min\{L,L^*\}]$, $|\kappa(t) - \kappa^*(t)|<\varepsilon$.
\end{corollary}
\begin{proof}
By lemma \ref{gluing2}, we can get analogous sequences $\{t_j^1\},\{t_j^2\}$. We argue by induction on $j$.\\
For $t\in [0, \min\{t_1^1,t_1^2\}]$, this is true by hypothesis. Assume without loss of generality that $t_1^1<t_1^2$. We have already shown that for all $\varepsilon>0$ we can take $\delta>0$ so that for all $t\in [t_1^1, t_1^2]$, $|\kappa^*(t) - \kappa(t_1^1)|<\varepsilon/2$, and $\alpha:=|t_1^1 - t_1^2|$ can be made arbitrarily small. $\kappa$ is continuous on a compact interval so it is uniformly continuous. Thus we can take $\alpha>0$ so that $|\kappa(t_1) - \kappa(t_2)|<\varepsilon/2$ for all $|t_1-t_2|<\alpha$. Thus $|\kappa^*(t)- \kappa(t)|<\varepsilon$ for all $t\in[t_1^1, t_1^2]$ as desired.\\
Now, assume the inductive hypothesis for $j$. We then have that $|\kappa^*(t)-\kappa(t)|\leq \varepsilon/2$ for all $t\in [0,\max\{t_j^1,t_j^2\}]$, and we can pick $\delta>0$ so that $|t_j^1 - t_j^2| <\alpha$ as described before. Assume without loss of generality that $t_j^1<t_j^2<t_{j+1}^1 < t_{j+1}^2$ for convenience. By hypothesis, for
$
t\in[t_j^2, t_{j+1}^1],
$
\begin{gather*}
    |\kappa^*(t) - \kappa(t)| = |\kappa^*(t) - \kappa(t - |t_j^1 - t_j^2|)\\ + \kappa(t - |t_j^1 - t_j^2|) -\kappa(t)|\\
    \leq |\kappa^*(t) - \kappa(t - |t_j^1 - t_j^2|)|\\
    + |\kappa(t - |t_j^1 - t_j^2|) -\kappa(t)|\\
    < \varepsilon/2 + \varepsilon/2 = \varepsilon.
\end{gather*}
\end{proof}
Thus if we can find a nice criteria for when we can find a partition, then we will be able to make a global statement on the relationship between the signature and the equivalence class of curves it corresponds to.

\section{Uniqueness Criteria}

We show the following so that we can extend theorem \ref{MainTheoremOpenCurves}, to the finite order vertex case.
\begin{lemma}\label{vertexSpecial}
Let $S^{(i)}$ be an injective phase portrait such that, for some $k\leq i,$ $\kappa_s^{(k)}(t)\neq 0$ for any $t$, and let $S_*^{(i)}$ be any other phase portrait. For all $\varepsilon>0$, there exists $\delta>0$ such that $d(S^{(i)},S_*^{(i)})<\delta$ implies $|\sigma^{(k)}(t) - \sigma_*^{(k)}(t)| < \varepsilon,\ \forall t\in [0, \ell:=\min\{L^*,L\}]$, $|L-L^*|<\varepsilon.$
\end{lemma}
\begin{proof}
 Fix $\varepsilon>0$. As the phase portrait determines the orientation, we can take $\delta>0$ so that $|\sigma_*^{(i)}(0) - \sigma^{(i)}(0)|<\varepsilon.$ What we now want to do is apply theorem \ref{MainTheoremOpenCurves} to the two coordinate projection $\{(\kappa_s^{(k-1)},\kappa_s^{(k)})\}$. As this set  satisfies the necessary hypotheses, we can restrict $\delta>0$ so that $|L-L^*|<\varepsilon$. Take $L^*>L$ without loss of generality, so that for all $t\in[0,L]$,
$$
|\kappa_s^{(k-1)}(t) - (\kappa^*)_s^{(k-1)}(t)|<\varepsilon.
$$
We now inductively show that this holds for all lower order derivatives. The base case has already been shown for $(k-1)$, so let us assume it holds for some $j<k-1$. We see then
\begin{gather*}
    |\kappa_s^{(j-1)}(t) - (\kappa^*)_s^{(j-1)}(t)| \\
    \leq |\kappa_s^{(j-1)}(0) - (\kappa^*)_s^{(j-1)}(0)|\\
    + \int_0^t |\kappa_s^{(j)}(s) - (\kappa^*)_s^{(j)}(s)|ds\\
    \leq \varepsilon + t\varepsilon \leq \varepsilon(1+L)
\end{gather*}
which can be made as small as desired. This completes our induction. For all $\varepsilon>0$, we can then take $\delta>0$ so that $|\kappa(t) - \kappa^*(t)|<\varepsilon$ on $[0,L]$, and $|L-L^*|<\varepsilon$.
\end{proof}
We are now prepared to prove the main result.
{\center{
\includegraphics[width=0.45\columnwidth]{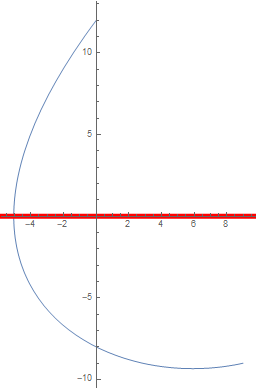}
\includegraphics[width=0.6\columnwidth]{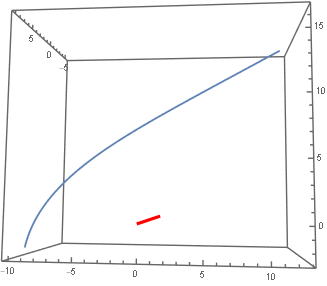}
\captionof{figure}{An example of a signature that has vertices (above) but $S^{(2)}$ (below) has no higher order vertices.}
}}

\medskip

\begin{theorem}
Let $S^{(i)}$ be a phase portrait that is injective and 
$$
S^{(i)}\cap \{(x,0,\dots, 0)\} = \emptyset.
$$
Then for any $S_*^{(i)}$, for all $\varepsilon>0$, there exists $\delta>0$ so that 
$$
d(S^{(i)}, S_*^{(i)})<\delta \Longrightarrow |\kappa^*(t)-\kappa(t)|<\varepsilon
$$
for $t\in[0,\ell:=\min\{L,L^*\},$ and $|L-L^*|<\varepsilon.$
\end{theorem}
\begin{proof}
It suffices to find a partition of $S^{(i)}$, as shown in corollary \ref{PartitionMeat}. For every $t\in[0,L]$, by hypothesis it must be that
$$
\sigma^{(i)}(t) = (u_1,\dots, u_{k_t},\dots,u_i)
$$
where $u_{k_t}\neq 0,$ $k_t>1$. Thus there exists $\delta_t>0$ so that $|\kappa_s^{(k_t)}(x)|> |\kappa_S^{(k_t)}(t)/2|$ on $(t-\delta_t,t+\delta_t)\cap [0,L]$. Because this interval is compact, we can then take a finite sub-cover of these sets, which we call $\{B_n\}_{n=1}^N.$ We take the left and right endpoints of the closures of each of these intervals, and order them so that
$$
0=t_0 \leq t_1 \leq t_2 \leq \dots \leq t_{2n}=L.
$$
By our construction, there exists some $k\geq 2$ on each of the intervals $[t_j,t_{j+1}]$ so that $\kappa_s^{(k)}\neq 0$. By lemma \ref{vertexSpecial}, this satisfies the requirements of a partition, and so we are done.
\end{proof}

In summary, we have shown the following.
\begin{theorem}
Let $\kappa(t):[0,L]\to \mathbb{R},\kappa^*(t):[0,L^*]\to \mathbb{R}$ be curvature functions of their defining curves $\Gamma,\Gamma^*$. Then for all $\varepsilon>0$, there exists $\delta>0$ so that if $|\kappa(t) - \kappa^*(t)| < \delta$ for all $t\in [0,\min\{L,L^*\}]$  and $|L-L^*|<\delta$ then there exists $g\in G$ such that $d(\Gamma, g\Gamma^*)<\varepsilon$.
\end{theorem}
\begin{corollary} \label{MainTheoremClosedCurves}
Let $S^{(i)}$ be a signature that is injective and
$$
S^{(i)}\cap \{(x,0,\dots, 0)\} = \emptyset.
$$
Then for any $S_*^{(i)}$, for all $\varepsilon>0$, there exists $\delta>0$ so that 
$$
d(S^{(i)}, S_*^{(i)})<\delta \Longrightarrow d(\Gamma,g\Gamma)<\varepsilon
$$
for some $g\in G.$
\end{corollary}

We now extend this to a uniqueness result on injective signatures first, then simple signatures. 
\begin{theorem}
Let $S^{(i)}$ be a phase portrait that is injective and 
$$
S^{(i)}\cap \{(x,0,\dots, 0)\} = \emptyset.
$$
Then there is a unique equivalence class of curves with respect to the group action $G$ that have this signature.
\end{theorem}
\begin{proof}
If any two curves have the same signature, by theorem \ref{MainTheoremOpenCurves}, then we see that 
$$
d(S^{(i)},S_*^{(i)})<\delta,\quad \forall \delta>0.
$$
Since we choose the same $g$ in our argument regardless of $\delta$, we see then that there exists $g\in G$ so that
$$
d(\Gamma, g\Gamma^*)<\varepsilon,\quad \forall \varepsilon >0
$$
thus $d(\Gamma, g\Gamma^*)=0\ \Longrightarrow \Gamma = g\Gamma^*$ so they are in the same equivalence class.
\end{proof}

For the next step we will need our proof to additionally hold for open curves/signatures, which we quickly verify.
\begin{lemma}
Let $S^{(i)}_0$ be a pre-compact signature that is injective and can be extended to a compact signature $S^{(i)}$ by taking its closure, and 
$$
d(S^{(i)}, \{(x,0,\dots, 0)\})>0.
$$
Then there is a unique equivalence class of curves with respect to the group action $G$ that have this signature.
\end{lemma}

\begin{proof}
To prove uniqueness in this case, it is easier instead to rely on the uniqueness of solutions to ODEs. Indeed, taking a partition as before, we can apply lemma \ref{sigDEQ} to solve for the function $\kappa^{(k-1)}$ explicitly. Because the initial values are determined by the location on the signature, we can uniquely solve for $\kappa$ on each of these intervals in the partition, and then use them together to define a global solution.
\end{proof}

\begin{theorem}
Let $S^{(i)}$ be a phase portrait that is simple and 
$$
S^{(i)}\cap \{(x,0,\dots, 0)\} = \emptyset.
$$
Then there is a unique equivalence class of curves with respect to the group action $G$ that have this signature.
\end{theorem}
\begin{proof}
The proof follows by the argument of theorem 2.
\end{proof}

This is in fact an equivalence. Indeed, at any such point $(x,0,\dots, 0)$, we can insert a section of constant curvature $x$ without altering smoothness. As all such curves will produce this same signature, we see that there is not a unique defining curve equivalence class.

\subsection{Metrics on Closed Curves}

For a closed curve $\gamma$, with minimal period $L$, and minimal period of the curvature $\ell$, we call its index of symmetry $i_S(\Gamma):= L/\ell$. It will be useful for us to think of our signature not as one set that is repeated after each revolution, but rather its injection into $\mathbb{R}^{i+1}$, given by
$$
S^{(i)}_p := \{(\kappa(t),\dots, \kappa^{(i)}(t), t)\}
$$
where $p:=\gamma(0).$
\begin{lemma}
Let $S^{(i)}\cap\{x,0,0,\dots,0\} = \emptyset.$
The set $\{t_n\}_{n=1}^{i_s(\Gamma)}$ forms a partition of $S^{(i)}_p$.
\end{lemma}
\begin{proof}
This follows immediately from corollary \ref{MainTheoremClosedCurves} as each segment $\sigma_p^{(i)}([t_i, t_{i+1}])$ meets the desired hypotheses.
\end{proof}

\begin{corollary}
Let $S^{(i)}$ be a simple signature and
$$
S^{(i)}\cap \{(x,0,\dots, 0)\} = \emptyset.
$$
Then for any $S_*^{(i)}$ with $\i_s(\Gamma^*)=i_s(\Gamma)$, for all $\varepsilon>0$, there exists $\delta>0$ so that 
$$
d(S^{(i)}, S_*^{(i)})<\delta \Longrightarrow d(\Gamma,g\Gamma)<\varepsilon
$$
for some $g\in G.$
\end{corollary}

{\centering \textbf{Acknowledgements}}

This work was done at the NCSU summer REU in Mathematics, generously funded by the NSA REU grant and NCSU Math Department, under the mentorship of Professor Irina Kogan and Eric Geiger. We are also thankful for discussions and ideas shared by group members Brooke Dippold and Jose Agudelo.


\begin{thebibliography}{}

\bibitem{Geiger} Eric Geiger and Irina Kogan, Non-congruent non-degenerate curves with identical signatures, Arxiv, 2019
\bibitem{Boutin} Mireille Boutin, Numerically invariant signature curves, Int. J. Computer vision 40 (2000),
235–248.
\bibitem{Bruckstein1} Alfred M. Bruckstein, Nir Katzir, Michael Lindenbaum, and Moshe Porat, Similarityinvariant signatures for partially occluded planar shapes, Int. J. Comput Vision 7 (1992),
no. 3, 271–285.
\bibitem{Bruckstein2} Alfred M. Bruckstein and Arun N. Netravali, On differential invariants of planar curves and
recognizing partially occluded planar shapes, Ann. Math. Artificial Intelligence 13 (1995),
no. 3-4, 227–250. MR 1335735
\bibitem{Calabi} Eugenio Calabi, Peter J. Olver, Chehrzad Shakiban, Allen Tannenbaum, and Steven Haker,
Differential and numerically invariant signature curves applied to object recognition, International Journal of Computer Vision 26 (1998), no. 2, 107–135.
\bibitem{Cartan} E. Cartan, Les problemes dequivalence, Oeuvres completes, II, pp. 1311-1334, GauthierVillars, Paris, 1953.
\bibitem{DeTurck} Dennis DeTurck, Herman Gluck, Daniel Pomerleano, and David Shea Vick, The four vertex
theorem and its converse, Notices Amer. Math. Soc. 54 (2007), no. 2, 192–207. MR 2285124
\bibitem{Grim} Anna Grim and Chehrzad Shakiban, Applications of signature curves to characterize
melanomas and moles, Applications of computer algebra, Springer Proc. Math. Stat., vol.
198, Springer, Cham, 2017, pp. 171–189. MR 3696633
\bibitem{Guggenheimer} Heinrich W. Guggenheimer, Differential geometry, McGraw-Hill Book Co., Inc., New YorkSan Francisco-Toronto-London, 1963. MR 0156266
31
\bibitem{Hickman} Mark S. Hickman, Euclidean signature curves, Journal of Mathematical Imaging and Vision
43 (2012), no. 3, 206–213.
\bibitem{Hoff} Daniel J. Hoff and Peter J. Olver, Extensions of invariant signatures for object recognition,
Journal of Mathematical Imaging and Vision 45 (2013), no. 2, 176–185.
\bibitem{Jigsaw} Automatic solution of jigsaw puzzles, J. Math. Imaging Vision 49 (2014), no. 1,
234–250. MR 3180965
\bibitem{Wolfgang} Wolfgang Kuhnel, Differential geometry, Student Mathematical Library, vol. 77, American
Mathematical Society, Providence, RI, 2015, Curves—surfaces—manifolds, Third edition
[of MR1882174], Translated from the 2013 German edition by Bruce Hunt, with corrections
and additions by the author. MR 3443721
\bibitem{Lee} John M. Lee, Introduction to smooth manifolds, second ed., Graduate Texts in Mathematics,
vol. 218, Springer, New York, 2013. MR 2954043
\bibitem{Musso} Emilio Musso and Lorenzo Nicolodi, Invariant signatures of closed planar curves, Journal
of Mathematical Imaging and Vision 35 (2009), no. 1, 68–85.
\bibitem{Olver} Peter J. Olver, Equivalence, invariants and symmetry, Cambridge University Press, 1995.


\end{thebibliography}
\end{document}